\documentclass[a4paper,12pt,reqno]{amsart}

\usepackage{amsmath}
\usepackage{amssymb}
\usepackage{amsfonts}
\usepackage{graphicx}
\usepackage[colorlinks]{hyperref}
\renewcommand\eqref[1]{(\ref{#1})} %Need with hyperref

\graphicspath{ {images/} }
\title[Global existence and blow-up of solutions]{Global existence and blow-up of solutions to porous medium equation and pseudo-parabolic equation, I. Stratified Groups}

\author[Michael Ruzhansky]{Michael Ruzhansky}
\address{\href{www.ruzhansky.org}{Michael Ruzhansky:}
	\endgraf
	Department of Mathematics: Analysis, Logic and Discrete Mathematics
	\endgraf
	Ghent University, Belgium
	\endgraf
	and
	\endgraf
	School of Mathematical Sciences
	\endgraf Queen Mary University of London 
	\endgraf
	United Kingdom
	\endgraf
	{\it E-mail address} {\rm Michael.Ruzhansky@ugent.be}
}

\author[Bolys Sabitbek]{Bolys Sabitbek}
\address{ \href{http://analysis-pde.org/bolys-sabitbek/}{Bolys Sabitbek:}
	\endgraf
	School of Mathematical Sciences
	\endgraf Queen Mary University of London 
	\endgraf
	United Kingdom
	\endgraf 
	and
	\endgraf 
	Al-Farabi Kazakh National University 
	\endgraf 
	Almaty,
	Kazakhstan 
	\endgraf
	{\it E-mail address} {\rm b.sabitbek@qmul.ac.uk}
}
\author[Berikbol Torebek]{Berikbol Torebek}
\address{{Berikbol Torebek:}
	\endgraf
	Department of Mathematics: Analysis, Logic and Discrete Mathematics 
	\endgraf
	Ghent University
	\endgraf
	Belgium
	\endgraf 
	and
	\endgraf 
	Institute of Mathematics and Mathematical Modeling 
	\endgraf
	Almaty,
	Kazakhstan
	\endgraf
	{\it E-mail address} {\rm berikbol.torebek@ugent.be}
}

\subjclass{35K92; 35B44, 35A01.}
\keywords{Blow-up, $p$-sub-Laplacian, porous medium equation, global solution, pseudo-parabloic equation, stratified groups}

\thanks{The first and second authors were supported by EPSRC grant EP/R003025/2. The first and third authors were also supported by FWO Odysseus 1 grant G.0H94.18N: Analysis and Partial Differential Equations and the Methusalem programme of the Ghent University Special Research Fund (BOF) (Grant number 01M01021).}

\newtheoremstyle{theorem}%name
{10pt}          % space above
{10pt}  % space below
{\sl}  % bofy font
{\parindent}     % ident - empty=no indent,  \parindent= paragraph indent
{\bf}  % thm head font
{. }    % punctuation after thm head
{ }    % space after thm head: `` ``=normal \newline=linebreak
{}     % thm head specification
\theoremstyle{theorem}

\numberwithin{equation}{section}
\theoremstyle{plain}
\newtheorem{thm}{Theorem}[section]

\newtheorem{lem}[thm]{Lemma}

\theoremstyle{definition}
\newtheorem{defn}[thm]{Definition}
\newtheorem{rem}[thm]{Remark}

\newtheoremstyle{defi}%name
{10pt}          % space above
{10pt}  % space below
{\rm}  % bofy font
{\parindent}     % ident - empty=no indent,  \parindent= paragraph indent
{\bf}  % thm head font
{. }    % punctuation after thm head
{ }    % space after thm head: `` ``=normal \newline=linebreak
{}     % thm head specification
\theoremstyle{defi}

\begin{document}
	\begin{abstract}
		In this paper, we prove a global existence and blow-up of the positive solutions to the initial-boundary value problem of the nonlinear porous medium equation and the nonlinear pseudo-parabolic equation on the stratified Lie groups. Our proof is based on the concavity argument and the Poincar\'e inequality, established in \cite{RS_JDE} for stratified groups.
	\end{abstract}
	\maketitle
	\tableofcontents
	\section{Introduction}
	The main purpose of this paper is to study the global existence and blow-up of the positive solutions to the initial-boundary problem of the nonlinear porous medium equation
	\begin{align}\label{main_eqn_p>2}
	\begin{cases}
	u_t -\mathcal{L}_p (u^m) = f(u), \,\,\, & x \in D,\,\, t>0, \\ 
	u(x,t)  =0,  \,\,\,& x\in \partial D,\,\, t>0, \\
	u(x,0)  = u_0(x)\geq 0,\,\,\, & x \in \overline{D}, 
	\end{cases}
	\end{align}
	and the nonlinear pseudo-parabolic equation
	\begin{align}\label{main_eqn_p}
	\begin{cases}
	u_t - \nabla_H \cdot (|\nabla_H u|^{p-2}\nabla_H u_t) -\mathcal{L}_pu = f(u), \,\,\, & x \in D,\,\, t>0, \\ 
	u(x,t)  =0,  \,\,\,& x\in \partial D,\,\, t>0, \\
	u(x,0)  = u_0(x)\geq 0,\,\,\, & x \in \overline{D},
	\end{cases}
	\end{align}
	where $m \geq 1$ and $p\geq 2$, $f$ is locally Lipschitz continuous on $\mathbb{R}$, $f(0)=0$, and such that $f(u)>0$ for $u>0$. Furthermore, we suppose that $u_0$ is a non-negative and non-trivial function in $C^1(\overline{D})$ with $u_0(x)=0$ on the boundary $\partial D$ for $p=2$ and in $L^{\infty}(D)\cap \mathring{S}^{1,p}(D)$ for $p>2$, respectively.

	\begin{defn}
		Let $\mathbb{G}$ be a stratified group. We say that an open set $D \subset \mathbb{G}$ is an admissible domain if it is bounded and if its boundary $\partial D$ is piecewise smooth and simple, that is, it has no self-intersections.
	\end{defn}
	
	Let $\mathbb{G}$ be a stratified group. Let $D \subset \mathbb{G}$ be an open set, then we define the functional spaces
	\begin{equation}\label{sobolev}
	S^{1,p}(D) =\{ u: D \rightarrow \mathbb{R}; u, |\nabla_{H} u| \in L^p(D) \}.
	\end{equation}
	We consider the following functional 
	\begin{equation*}
	\mathcal{J}_p(u):= \left( \int_{D} |\nabla_{H} u(x)|^p dx\right)^{\frac{1}{p}}.
	\end{equation*}
	Thus, the functional class $\mathring{S}^{1,p}(D)$ can be defined as the completion of $C_0^1(D)$ in the norm generated by $\mathcal{J}_p$, see e.g. \cite{CDG1993}.
	
	A Lie group $\mathbb{G}=(\mathbb{R}^{n},\circ)$ is called a stratified (Lie) group
	if it satisfies the following conditions:
	
	(a) For some integer numbers $N_1+N_{2}+...+N_{r}=n$,
	the decomposition $\mathbb{R}^{n}=\mathbb{R}^{N_1}\times\ldots\times\mathbb{R}^{N_{r}}$ is valid, and
	for any $\lambda>0$ the dilation 
	$$\delta_{\lambda}(x):=(\lambda x',\lambda ^{2}x^{(2)},\ldots,\lambda^{r}x^{(r)})$$
	is an automorphism of $\mathbb{G}.$ Here $x'\equiv x^{(1)}\in \mathbb{R}^{N_1}$ and $x^{(k)}\in \mathbb{R}^{N_{k}}$ for $k=2,\ldots,r.$
	
	(b) Let $N_1$ be as in (a) and let $X_{1},\ldots,X_{N_1}$ be the left-invariant vector fields on $\mathbb{G}$ such that
	$X_{k}(0)=\frac{\partial}{\partial x_{k}}|_{0}$ for $k=1,\ldots,N_1.$ Then the H\"ormander rank condition must be satisfied, that is, 
	$${\rm rank}({\rm Lie}\{X_{1},\ldots,X_{N_1}\})=n,$$
	for every $x\in\mathbb{R}^{n}.$ 
	
	Then, we say that the triple $\mathbb{G}=(\mathbb{R}^{n},\circ, \delta_{\lambda})$ is a stratified (Lie) group.
	
	Recall that the standard Lebesgue measure $dx$ on $\mathbb R^{n}$ is the Haar measure for $\mathbb{G}$ (see e.g. \cite{FR}, \cite{RS_book}).
	The left-invariant vector field $X_{j}$ has an explicit form: 
	\begin{equation}\label{Xk0}
	X_{k}=\frac{\partial}{\partial x'_{k}}+
	\sum_{l=2}^{r}\sum_{m=1}^{N_{l}}a_{k,m}^{(l)}(x',...,x^{(l-1)})
	\frac{\partial}{\partial x_{m}^{(l)}},
	\end{equation}
	see e.g. \cite{RS_book}.
	The following notations are used throughout this paper:
	$$\nabla_{H}:=(X_{1},\ldots, X_{N_1})$$
	for the horizontal gradient, and
	\begin{equation}\label{pLap}
	\mathcal{L}_{p}f:=\nabla_{H}\cdot(|\nabla_{H}f|^{p-2}\nabla_{H}f),\quad 1<p<\infty,
	\end{equation}
	for the $p$-sub-Laplacian. When $p=2$, that is,
	the second order differential operator
	\begin{equation}\label{sublap}
	\mathcal{L}=\sum_{k=1}^{N_1}X_{k}^{2},
	\end{equation}
	is called the sub-Laplacian on $\mathbb{G}$.
	The sub-Laplacian $\mathcal{L}$ is a left-invariant homogeneous hypoelliptic differential operator and it is known that $\mathcal{L}$ is elliptic if and only if the step of $\mathbb{G}$ is equal to 1.
	
	One of the important examples of the nonlinear parabolic equations is the porous medium equation, which describes widely processes involving fluid flow, heat transfer or diffusion, and its other applications in different fields such as mathematical biology, lubrication, boundary layer theory, and etc. Existence and nonexistence of solutions to problem \eqref{main_eqn_p>2} for the reaction term $u^m$ in the case $m=1$ and $m>1$ have been actively investigated by many authors, for example, \cite{Ball, Band-Brun, Chen-Fila-Guo, Ding-Hu, Deng-Levine, Gal-Vaz-97, Gr-Mu-Po-13, Hayakawa, Ia-San-14, Ia-San-19, Levine90, LP1, ST-21,  Sam-Gal-Ku-Mik, Souplet}, Grillo, Muratori and Punzo considered fractional porous medium equation \cite{Gr-Mu-Pu-1, Gr-Mu-Pu-2}, and it was also considered in the setting of Cartan-Hadamard manifolds \cite{Gr-Mu-Pu-3}. By using the concavity method, Schaefer \cite{Sch09} established a condition on the initial data of a Dirichlet type initial-boundary value problem for the porous medium equation with a power function reaction term when blow-up of the solution in finite time occurs and a global existence of the solution holds.  We refer for more details to Vazquez's book \cite{Vaz} which provides a systematic presentation of the mathematical theory of the porous medium equation.
	
	The energy for the isotropic material  can be modeled by a pseudo-parabolic equation \cite{CG68}. Some wave processes \cite{BBM72}, filtration of the two-phase flow in porous media with the dynamic capillary pressure \cite{Baren} are also modeled by pseudo-parabolic equations. The global existence and finite-time blow-up for the solutions to pseudo-parabolic equations in bounded and unbounded domains have been studied by many researchers, for example, see \cite{Korpusov1, Korpusov2, Long, Luo, Peng, Xu1, Xu2, Xu3} and the references therein.
	
	In \cite{PohVer}, Veron and Pohozaev have obtained blow-up results for the following semi-linear diffusion equation on the Heisenberg groups
	$$
	\frac{\partial u(x,t)}{\partial t}-\mathcal{L} u(x,t)=|u(x,t)|^{p},\,\,\,\,\,(x,t)\in \mathbb{H} \times(0,+\infty).
	$$
	Also, blow-up of the solutions to the semi-linear diffusion and pseudo-parabolic equations on the Heisenberg groups was derived in \cite{AAK1, AAK2, DL1, JKS1, JKS2}.
	In addition, in \cite{RY} the authors found the Fujita exponent on general unimodular Lie groups.

	In some of our considerations a crucial role is played by
	\begin{itemize}
		\item  The condition 
		\begin{equation}\label{cond-f}
		\alpha F(u) \leq u^m f(u) + \beta u^{pm} +\alpha\gamma,\,\,\, u>0,
		\end{equation}
		where 
		\begin{equation*}
		F(u)=\frac{pm}{m+1}\int_{0}^{u}s^{m-1}f(s)ds, \,\,\, m\geq 1,
		\end{equation*}
		introduced by Chung-Choi \cite{Chung-Choi} for a parabolic equation. We will deal with several variants of such condition.
		\item The Poincar\'e inequality established by the first author and Suragan in \cite{RS_JDE} for stratified groups: 
		\begin{lem}\label{lem1}
			Let $D \subset \mathbb{G}$ be an admissible domain with $N_1$ being the dimension of the first stratum. Let $1<p<\infty$ with $p\neq N_1$. For every function $u \in C_0^{\infty}(D \backslash \{x'=0\})$ we have 
			\begin{equation}
			\int_{D} |\nabla_{H} u|^p dx \geq \frac{|N_1-p|^p}{(pR)^p} \int_{D} |u|^p dx,
			\end{equation}
			where $R = \sup_{x \in D} |x'|$.
		\end{lem}
	\end{itemize}
	
	Note that it is possible to interpret the constant $\frac{|N_1-p|^p}{(pR)^p}$ as a measure of the size of the domain $D$. Then $\beta$ in \eqref{cond-f} is dependent on the size of the domain $D$.

	Our paper is organised so that we discuss the existence and nonexistence of positive solutions to the nonlinear porous medium equation in Section \ref{sec1} and the nonlinear pseudo-parabolic equation in Section \ref{sec2}. 
	
	\section{Nonlinear porous medium equation}\label{sec1}
	In this section, we prove the global solutions and blow-up phenomena of the initial-boundary value problem \eqref{main_eqn_p>2}.
	\subsection{Blow-up solutions of the nonlinear porous medium equation} We start with the blow-up properly. 
	\begin{thm}\label{thm_p>2}
		Let $\mathbb{G}$ be a stratified group with $N_1$ being the dimension of the first stratum. Let $D \subset \mathbb{G}$ be an admissible domain. Let $2\leq  p<\infty$ with $p\neq N_1$. 
		
		Assume that function $f$ satisfies
		\begin{equation}\label{condt_p}
		\alpha F(u) \leq u^{m} f(u) + \beta u^{pm} +\alpha\gamma,\,\,\, u>0,
		\end{equation}
		where 
		\begin{equation*}
		F(u)=\frac{pm}{m+1}\int_{0}^{u}s^{m-1}f(s)ds, \,\,\, m\geq 1,
		\end{equation*}			
		for some
		\begin{equation*}
		\gamma >0,\,\, 0<\beta\leq \frac{|N_1-p|^p}{(pR)^p}\frac{(\alpha - m -1)}{m+1}\,\,\, \text{ and } \,\, \alpha>m+1,
		\end{equation*} 
		where $R = \sup_{x \in D} |x'|$ and $x=(x',x'')$ with $x'$ being in the first stratum.
		Let  $u_0 \in L^{\infty}(D)\cap\mathring{S}^{1,p}(D)$ satisfy the inequality 	 \begin{equation}\label{J(1)}
		J_0:=   - \frac{1}{m+1} \int_{D} |\nabla_{H} u^m_0(x)|^p dx +  \int_{D} (F(u_0(x))-\gamma) dx >0.
		\end{equation} 
		Then any positive solution $u$ of \eqref{main_eqn_p>2} blows up in finite time $T^*,$ i.e., there exists 
		\begin{equation}\label{T}
		0<T^*\leq \frac{M}{\sigma \int_{D}u_0^{m+1}(x)dx},
		\end{equation}
		such that 
		\begin{equation}
		\lim_{t\rightarrow T^*} \int_{0}^t \int_{D} u^{m+1} (x,\tau) dx d\tau = +\infty,
		\end{equation}
		where $M>0$ and $\sigma = \frac{\sqrt{pm\alpha}}{m+1}-1>0$.
		In fact, in \eqref{T}, we can take 
		\begin{equation*}
		M = \frac{(1+\sigma)(1+1/\sigma)(\int_D u_0^{m+1}(x)dx)^2}{\alpha (m+1)J_0}.
		\end{equation*}
	\end{thm}
	\begin{rem}
		Note that condition on nonlinearity \eqref{condt_p} includes the following cases: 
		\begin{itemize}
			\item[1.] Philippin and Proytcheva \cite{PhP-06} used the condition 
			\begin{equation}
			(2+\epsilon)F(u) \leq u f(u), \,\,\, u>0,
			\end{equation}
			where $\epsilon >0$. It is a special case of an abstract condition by Levine and Payne \cite{LP2}. 
			\item[2.] Bandle and Brunner \cite{Band-Brun} relaxed this condition as follows
			\begin{equation}
			(2+\epsilon)F(u) \leq u f(u) + \gamma, \,\,\, u>0,
			\end{equation}
			where $\epsilon >0$ and $\gamma >0$.
		\end{itemize}
		These cases were established on the bounded domains of the Euclidean space, and it is a new result on the stratified groups.
	\end{rem}
	\begin{proof}[Proof of Theorem \ref{thm_p>2}]
		Assume that $u(x,t)$ is a positive solution of \eqref{main_eqn_p>2}.	We use the concavity method for showing the blow-up phenomena. We introduce the functional
		\begin{align}\label{J1}
		J(t) := -\frac{1}{m+1} \int_{D} |\nabla_{H} u^m(x,t)|^p dx + \int_{D} (F(u(x,t))-\gamma) dx, 
		\end{align}
		and by \eqref{J(1)} we have
		\begin{align}\label{J0}
		J(0) = -\frac{1}{m+1} \int_{D} |\nabla_{H} u^m_0(x)|^p dx + \int_{D} (F(u_0(x))-\gamma) dx>0.
		\end{align}
		Moreover, $J(t)$ can be written in the following form
		\begin{equation}\label{Fp}
		J(t) = J(0) + \int_{0}^t \frac{d J(\tau)}{d\tau}d\tau,
		\end{equation}
		where
		\begin{align*}
		\int_{0}^t \frac{d J(\tau)}{d\tau}d\tau  &=  - \frac{1}{m+1} \int_{0}^t\int_{D} \frac{d}{d\tau}|\nabla_{H} u^m(x,\tau)|^p dx d\tau  + \int_{0}^t \int_{D} \frac{d}{d\tau} (F(u(x,\tau))-\gamma) dx d\tau\\
		& =   - \frac{p}{m+1} \int_{0}^t\int_{D} |\nabla_{H} u^m(x,\tau)|^{p-2} \nabla_{H} u^m \cdot \nabla_{H} (u^m(x,\tau))_{\tau} dx d\tau\\
		& +  \int_{0}^t \int_{D} F_u(u(x,\tau)) u_{\tau}(x,\tau) dx d\tau\\
		& = \frac{p}{m+1} \int_{0}^t \int_{D} [\mathcal{L}_p (u^m )+ f(u)](u^m(x,\tau))_{\tau} dx d\tau \\
		& = \frac{pm}{m+1}\int_{0}^t \int_{D} u^{m-1}(x,\tau) u_{\tau}^2(x,\tau) dx d\tau.
		\end{align*}

		Define
		\begin{equation*}
		E(t) = \int_{0}^t \int_{D} u^{m+1}(x, \tau) dx d\tau + M, \,\, t\geq 0,
		\end{equation*}
		with $M>0$ to be chosen later. Then the first derivative with respect $t$ of $E(t)$ gives
		\begin{equation*}
		E'(t)=\int_{D} u^{m+1}(x,t) dx = (m+1)\int_{D} \int_{0}^t u^{m}(x,\tau) u_{\tau}(x,\tau) d\tau dx + \int_{D} u^{m+1}_0(x) dx.
		\end{equation*}
		By applying  \eqref{condt_p}, Lemma \ref{lem1} and $0<\beta\leq \frac{|N_1-p|^p}{(pR)^p}\frac{(\alpha - m -1)}{m+1}$, we estimate the second derivative of $E(t)$ as follows
		\begin{align*}
		E''(t) &=(m+1) \int_{D} u^{m}(x,t) u_t(x,t) dx\\
		& = -(m+1) \int_{D} |\nabla_{H} u^m(x,t)|^p dx + (m+1)\int_{D} u^{m}(x,t)f(u(x,t))dx
		\\
		& \geq - (m+1) \int_{D} |\nabla_{H} u^m(x,t)|^p dx + (m+1) \int_{D} \left[ \alpha F(u(x,t)) -\beta u^{pm} (x,t) -\alpha \gamma \right] dx\\
		& =\alpha (m+1) \left[ -\frac{1}{m+1} \int_{D} |\nabla_{H} u^m(x,t)|^p dx + \int_{D} (F(u(x,t))-\gamma) dx \right]  \\
		&+ (\alpha-m-1) \int_{D} |\nabla_{H} u^m(x,t)|^p dx -\beta (m+1) \int_{D} u^{pm}(x, t) dx \\
		& \geq \alpha(m+1) \left[ -\frac{1}{m+1} \int_{D} |\nabla_{H} u^m(x,t)|^p dx + \int_{D} (F(u(x,t))-\gamma) dx \right]  \\
		& + \left[ \frac{|N_1-p|^p}{(pR)^p}(\alpha - m-1) - \beta (m+1) \right]\int_{D} u^{pm}(x, t) dx\\
		&  \geq \alpha (m+1)\left[ -\frac{1}{m+1} \int_{D} |\nabla_{H} u^m(x,t)|^p dx + \int_{D} (F(u(x,t))-\gamma) dx \right]\\
		& = \alpha (m+1)J(t)\\
		& = \alpha (m+1) J(0) + p\alpha m \int_{0}^t \int_{D} u^{m-1}(x,\tau) u_{\tau}^2(x,\tau) dx d\tau.
		\end{align*}
		By employing the H\"older and Cauchy-Schwarz inequalities, we obtain the estimate for $[E'(t)]^2$ as follows
		\begin{align*}
		[E'(t)]^2&\leq  (1+\delta)\left( \int_{D} \int_{0}^t (u^{m+1}(x,\tau))_{\tau} d\tau dx  \right)^2 + \left( 1+ \frac{1}{\delta}\right)\left( \int_{D} u_0^{m+1}(x)dx \right)^2 \\
		& = (m+1)^2(1+\delta)\left( \int_{D} \int_{0}^t u^{m}(x,\tau)  u_{\tau}(x,\tau)dx d\tau \right)^2   + \left( 1+ \frac{1}{\delta}\right)\left( \int_{D} u_0^{m+1}(x)dx \right)^2\\
		& = (m+1)^2(1+\delta)\left( \int_{D} \int_{0}^t u^{(m+1)/2 + (m-1)/2}(x,\tau)  u_{\tau}(x,\tau)dx d\tau \right)^2   \\
		&+ \left( 1+ \frac{1}{\delta}\right)\left( \int_{D} u_0^{m+1}(x)dx \right)^2
		\end{align*}
		\begin{align*}
		& \leq (m+1)^2(1+\delta)\left( \int_{D} \left(\int_{0}^t u^{m+1} d\tau\right)^{1/2}\left( \int_{0}^t u^{m-1} u_{\tau}^2(x,\tau)d\tau\right)^{1/2} dx \right)^2   \\
		&+ \left( 1+ \frac{1}{\delta}\right)\left( \int_{D} u_0^{m+1}(x)dx \right)^2 \\
		& \leq (m+1)^2(1+\delta) \left(\int_{0}^t \int_{D} u^{m+1} dx d\tau\right)\left( \int_{0}^t \int_{D} u^{m-1} u_{\tau}^2(x,\tau)dxd\tau \right)   \\
		&+ \left( 1+ \frac{1}{\delta}\right)\left( \int_{D} u_0^{m+1}(x)dx \right)^2,
		\end{align*}
		for arbitrary $\delta>0$. So we have
		\begin{equation}\label{eq-E2}
		\small		[E'(t)]^2\leq (m+1)^2(1+\delta) \left(\int_{0}^t \int_{D} u^{m+1} dx d\tau\right)\left( \int_{0}^t \int_{D} u^{m-1} u_{\tau}^2dxd\tau \right) 
		+ \left( 1+ \frac{1}{\delta}\right)\left( \int_{D} u_0^{m+1}dx \right)^2.
		\end{equation}
		The previous estimates together with $\sigma=\delta= \frac{\sqrt{pm\alpha}}{m+1}-1>0$ where positivity comes from $\alpha > m+1$, imply
		\begin{align*}
		&E''(t) E (t) - (1+\sigma) [E'(t)]^2\\
		&\geq  \alpha M(m+1) \left[ -\frac{1}{m+1} \int_{D} |\nabla_{H} u^m_0|^p dx + \int_{D} (F(u_0)-\gamma)dx\right] 
		\\&+ pm\alpha \left( \int_{0}^t\int_{D} u^{m+1}(x, \tau) dx d\tau \right) \left(\int_{0}^t \int_{D} u_{\tau}^2(x,\tau) u^{m-1}(x,\tau) dx d\tau  \right)\\
		&- (m+1)^2(1+\sigma) (1+\delta) \left(\int_{0}^t \int_{D} u^{m+1} dx d\tau\right)\left( \int_{0}^t \int_{D} u^{m-1} u_{\tau}^2(x,\tau)dxd\tau \right)\\
		& - (1+\sigma)\left( 1+ \frac{1}{\delta}\right)\left( \int_{D} u_0^{m+1}(x)dx \right)^2 \\
		&\geq  \alpha M(m+1) J(0) - (1+\sigma)\left( 1+ \frac{1}{\delta}\right)\left( \int_{D} u_0^{m+1}(x)dx \right)^2.
		\end{align*}	
		By assumption $J(0)>0$, thus if we select 
		$$M =\frac{(1+\sigma)\left( 1+ \frac{1}{\delta}\right)\left( \int_{D} u_0^{m+1}(x)dx \right)^2}{\alpha (m+1) J(0)}, $$ 
		that gives
		\begin{equation}
		E''(t) E (t) - (1+\sigma) (E'(t))^2 \geq 0.
		\end{equation}
		We can see that the above expression for $t\geq 0$ implies 
		\begin{equation*}
		\frac{d}{dt} \left[ \frac{E'(t)}{E^{\sigma+1}(t)} \right] \geq 0  \Rightarrow 	\begin{cases}
		E'(t) \geq \left[ \frac{E'(0)}{E^{\sigma+1}(0)} \right] E^{1+\sigma}(t),\\
		E(0)=M.
		\end{cases}
		\end{equation*}
		Then for $\sigma = \frac{\sqrt{pm\alpha}}{m+1}-1>0$, we arrive at 
		\begin{align*}
		- \frac{1}{\sigma} \left[ E^{-\sigma}(t) - E^{-\sigma}(0)  \right] \geq \frac{E'(0)}{E^{\sigma+1}(0)} t,
		\end{align*}
		and some rearrangements with $E(0)=M$ give
		\begin{equation*}
		E(t) \geq \left( \frac{1}{M^{\sigma}}-\frac{ \sigma \int_{D} u^{m+1}_0(x)dx }{M^{\sigma+1}} t\right)^{-\frac{1}{\sigma}}.
		\end{equation*}
		Then the blow-up time $T^*$ satisfies 
		\begin{equation*}
		0<T^*\leq \frac{M}{\sigma \int_{D} u_0^{m+1}dx}.
		\end{equation*}
		That completes the proof.
	\end{proof}	
	\subsection{Global existence for the nonlinear porous medium equation} We now show that under some assumptions, if a positive solution to \eqref{main_eqn_p>2} exists, its norm is globally controlled.  
	\begin{thm}\label{thm_GEp}
		Let $\mathbb{G}$ be a stratified group with $N_1$ being the dimension of the first stratum. Let $D \subset \mathbb{G}$ be an admissible domain. Let $2 \leq  p<\infty$ with $p\neq N_1$.
		
		Assume that
		\begin{equation}\label{global_cond-p}
		\alpha F(u) \geq u^{m} f(u) + \beta u^{pm} +\alpha\gamma, \,\,\, u>0,
		\end{equation}
		where
		\begin{equation*}
		F(u)=\frac{pm}{m+1}\int_{0}^{u}s^{m-1}f(s)ds, \,\,\, m\geq 1,
		\end{equation*}
		for some
		\begin{equation*}
		\gamma \geq 0, \,\, \alpha \leq 0 \,\,\, \text{ and } \,\,\, \beta \geq \frac{|N_1-p|^p}{(pR)^p}\frac{( \alpha-m-1 )}{m+1},
		\end{equation*}
		where $R = \sup_{x \in D} |x'|$ and $x=(x',x'')$ with $x'$ being in the first stratum.
		
		Assume also that $u_0 \in L^{\infty}(D)\cap \mathring{S}^{1,p}(D)$ satisfies inequality 
		\begin{equation}\label{J(0)}
		J_0:= \int_{D} (F(u_0(x))-\gamma) dx - \frac{1}{m+1} \int_{D} |\nabla_{H} u^m_0(x)|^p dx>0.
		\end{equation}
		If $u$ is a positive local solution of problem \eqref{main_eqn_p>2}, then it is global and satisfies the following estimate:
		\begin{equation*}
		\int_D u^{m+1}(x,t) dx \leq  \int_D u^{m+1}_0(x)dx.
		\end{equation*}
		
	\end{thm}
	\begin{proof}[Proof of Theorem \ref{thm_GEp}]
		Recall from the proof of Theorem \ref{thm_p>2}, the functional
		
		\begin{align*}
		J(t) &:= -\frac{1}{m+1} \int_{D} |\nabla_{H} u^m(x,t)|^p dx + \int_{D} (F(u(x,t))-\gamma) dx\\
		& = J_0 + \frac{pm}{m+1}\int_{0}^t \int_{D} u^{m-1}(x,\tau) u_{\tau}^2(x,\tau) dx d\tau.
		\end{align*}
		
		Let us define
		
		\begin{equation*}
		\mathcal E(t) = \int_{D} u^{m+1}(x,t)dx.
		\end{equation*}
		By applying \eqref{global_cond-p}, Lemma \ref{lem1} and $\beta \geq \frac{|N_1-p|^p}{(pR)^p}\frac{( \alpha-m-1 )}{m+1}$, respectively, one finds 
		\begin{align*}
		\mathcal	E'(t) &=(m+1) \int_{D} u^{m}(x,t) u_t(x,t) dx\\
		& = (m+1) \left[\int_{D} u^{m}(x,t) \nabla_{H} \cdot (|\nabla_{H} u^m(x,t)|^{p-2}\nabla_{H} u^m(x,t)) + \int_{D} u^{m}(x,t)f(u(x, t))dx\right] \\
		& = (m+1) \left[-\int_{D} |\nabla_{H} u^m(x,t)|^p dx + \int_{D} u^{m}(x,t)f(u(x,t))dx\right]
		\\
		& \leq (m+1) \left[-\int_{D} |\nabla_{H} u^m(x,t)|^p dx +  \int_{D} \left[ \alpha F(u(x,t)) -\beta u^{pm} (x,t) -\alpha \gamma \right] dx\right]\\
		& =\alpha (m+1) \left[ -\frac{1}{m+1} \int_{D} |\nabla_{H} u^m(x,t)|^p dx + \int_{D} (F(u(x,t))-\gamma) dx \right]  \\
		&- (m+1-\alpha) \int_{D} |\nabla_{H} u^m(x,t)|^p dx -\beta (m+1) \int_{D} u^{pm}(x, t) dx\\
		& \leq \alpha (m+1)\left[ -\frac{1}{m+1} \int_{D} |\nabla_{H} u^m(x,t)|^p dx + \int_{D} (F(u(x,t))-\gamma) dx \right]  \\
		& - \left[ \frac{|N_1-p|^p}{(pR)^p}(m+1-\alpha ) + \beta (m+1) \right]\int_{D} u^{pm}(x, t) dx\\
		&  \leq \alpha (m+1)\left[ -\frac{1}{m+1} \int_{D} |\nabla_{H} u^m(x,t)|^2 dx + \int_{D} (F(u(x,t))-\gamma) dx \right] \\
		&=\alpha (m+1)J(t).
		\end{align*}
		We can rewrite $\mathcal E'(t)$ by using \eqref{Fp} and $\alpha \leq 0$ as follows
		\begin{align}
		\mathcal E'(t) \leq \alpha (m+1) J(0) + p\alpha m \int_{0}^t \int_{D} u^{m-1}(x,\tau) u_{\tau}^2(x,\tau) dx d\tau \leq 0.
		\end{align}
		That gives 
		\begin{equation*}
		\mathcal E(t) \leq \mathcal E(0).
		\end{equation*}
		This completes the proof of Theorem \ref{thm_GEp}.
	\end{proof}

	\section{Nonlinear pseudo-parabolic equation}\label{sec2}
	In this section, we prove the global solutions and blow-up phenomena of the initial-boundary value problem \eqref{main_eqn_p}.

	\subsection{Blow-up phenomena for the pseudo-parabolic equation} We start with conditions ensuring the blow-up of solutions in finite time. 
	\begin{thm}\label{thm_p>21}
		Let $\mathbb{G}$ be a stratified group with $N_1$ being the dimension of the first stratum. Let $D \subset \mathbb{G}$ be an admissible domain. Let $2\leq p<\infty$ with $p\neq N_1$. 
		
		Assume that
		\begin{equation}\label{hyp-blow-p}
		\alpha F(u) \leq u f(u) + \beta u^{p} +\alpha\gamma,\,\,\, u>0,
		\end{equation}
		where 
		\begin{equation*}
		F(u) = \int_{0}^u f(s)ds,
		\end{equation*}
		for some
		\begin{align}
		\alpha>p \,\, &\text{ and } \,\,	0<\beta\leq \frac{|N_1-p|^p}{(pR)^p}\frac{(\alpha - p)}{p}, \\
		\gamma>0 \,\,&\text{ and }\,\, R = \sup_{x \in D} |x'|. \nonumber
		\end{align}
		Assume also that $u_0 \in L^{\infty}(D)\cap \mathring{S}^{1,p}(D)$ satisfies
		\begin{equation}
		\mathcal{F}_0:= -\frac{1}{p} \int_{D} |\nabla_H u_0(x)|^p dx + \int_{D} (F(u_0(x))-\gamma) dx >0.
		\end{equation}
		Then any positive solution $u$ of \eqref{main_eqn_p} blows up in finite time $T^*,$ i.e., there exists 
		\begin{equation}
		0<T^*\leq \frac{M}{\sigma \int_{D} u_0^2 + \frac{2}{p} |\nabla_H u_0|^p dx},
		\end{equation}
		such that
		\begin{equation}
		\lim_{t\rightarrow T^*}\int_{0}^t \int_{D} [u^{2} + \frac{2}{p}|\nabla_H u|^p ]dx d\tau = +\infty,
		\end{equation}
		where $\sigma=\sqrt{\frac{\alpha}{2}}-1>0$ and 
		\begin{equation*}
		M = \frac{(1+\sigma)\left( 1+ \frac{1}{\sigma}\right)\left( \int_{D} u^{2}_0 + \frac{2}{p}|\nabla_H u_0|^p dx \right)^2}{2\alpha \mathcal{F}_0}.
		\end{equation*}
	\end{thm}

	\begin{proof}[Proof of Theorem \ref{thm_p>21}]
		The proof is based on a concavity method. The main idea is to show that $[E^{-\sigma}_p(t)]''\leq 0$ which means that $E^{-\sigma}_p(t)$ is a concave function, for $E_p(t)$ defined below.
		
		Let us introduce some notations: 
		\begin{align*}
		\mathcal{F}(t) := -\frac{1}{p} \int_{D} |\nabla_H u(x,t)|^p dx + \int_{D} (F(u(x,t))-\gamma) dx, 
		\end{align*}
		and
		\begin{align*}
		\mathcal{F}(0) := -\frac{1}{p} \int_{D} |\nabla_H u_0(x)|^p dx + \int_{D} (F(u_0(x))-\gamma) dx,
		\end{align*}
		with 
		\begin{equation*}
		F(u)=\int_{0}^{u}f(s)ds.
		\end{equation*}
		We know that
		\begin{equation}\label{F-p}
		\mathcal{F}(t) = \mathcal{F}(0) + \int_{0}^t \frac{d \mathcal{F}(\tau)}{d\tau}d\tau,
		\end{equation}
		where
		\begin{align*}
		\int_{0}^t \frac{d \mathcal{F}(\tau)}{d\tau}d\tau  &=  - \frac{1}{p} \int_{0}^t\int_{D} \frac{d}{d\tau}|\nabla_H u|^p dx d\tau  + \int_{0}^t \int_{D} \frac{d}{d\tau} (F(u)-\gamma) dx d\tau\\
		& =  -  \int_{0}^t\int_{D} |\nabla_H u|^{p-2} \nabla u \cdot \nabla_H u_{\tau} dx d\tau +  \int_{0}^t \int_{D} F_u(u) u_{\tau} dx d\tau\\
		& =  \int_{0}^t \int_{D} [\mathcal{L}_p u + f(u)]u_{\tau} dx d\tau \\
		& =  \int_{0}^t \int_{D} u_{\tau}^2 - u_{\tau} \nabla_H \cdot (|\nabla_H u|^{p-2} \nabla_H u_{\tau})dx d\tau\\
		& = \int_{0}^t \int_{D}  u_{\tau}^2  + |\nabla_H u|^{p-2}|\nabla_H u_{\tau}|^2dx d\tau.
		\end{align*}
		
		Let us define
		\begin{align*}
		E_p(t) &:= \int_{0}^t \int_{D}[ u^{2} + \frac{2}{p}|\nabla_H u|^p] dx d\tau + M, \,\, t\geq 0,
		\end{align*}
		with a positive constant $M>0$ to be chosen later. Then
		\begin{align}\label{eq-E}
		E'_p(t) = \int_{D}[ u^2 + \frac{2}{p} |\nabla_H u|^p ]dx = \int_{0}^t \frac{d}{d\tau}\int_{D} [u^2 + \frac{2}{p}|\nabla_H u|^p]dx d\tau + \int_{D} u^{2}_0 + \frac{2}{p}|\nabla_H u_0|^p dx. 
		\end{align}
		Now we estimate $E''_p(t)$ by using assumption \eqref{hyp-blow-p} and integration by parts, that gives
		\begin{align*}
		E''_p(t) &= 2\int_{D} u u_t dx + \frac{2}{p}\int_{D} (|\nabla_H u|^p)_t dx\\
		& =2 \int_{D}  [ u\mathcal{L}_pu+  u \nabla_H \cdot (|\nabla_H u|^{p-2}\nabla_H  u_t) + uf(u)]dx + \frac{2}{p}\int_{D} (|\nabla_H u|^p)_t dx\\
		& = -2 \int_{D} [|\nabla_H u|^p + |\nabla_H u|^{p-2}\nabla_H u \cdot \nabla_H u_t ] dx + 2\int_{D} uf(u)dx + \frac{2}{p}\int_{D} (|\nabla_H u|^p)_t dx
		\\
		& \geq - 2 \int_{D} |\nabla_H u|^p dx + 2 \int_{D} \left[ \alpha F(u) -\beta u^{p}  -\alpha \gamma \right] dx\\
		& =2 \alpha \left[ -\frac{1}{p} \int_{D} |\nabla_H u|^p dx + \int_{D} (F(u)-\gamma) dx \right]  \\
		&+ \frac{2(\alpha-p)}{p} \int_{D} |\nabla_H u|^p dx -2\beta \int_{D} u^{p} dx.
		\end{align*}
		Next we apply Lemma \ref{lem1}, which gives
		\begin{align*}
		& \geq 2\alpha \left[ -\frac{1}{p} \int_{D} |\nabla_H u|^p dx + \int_{D} (F(u)-\gamma) dx \right]  \\
		& + 2\left[\frac{|N_1-p|^p}{(pR)^p}\frac{(\alpha -p)}{p} - \beta \right]\int_{D} u^{p} dx\\
		&  \geq 2\alpha\left[ -\frac{1}{p} \int_{D} |\nabla_H u|^p dx + \int_{D} (F(u)-\gamma) dx \right]\\
		& = 2\alpha \mathcal{F}(t),
		\end{align*}
		with $\mathcal{F}(t)$ as in \eqref{F-p}, then $E''_p(t)$ can be rewritten in the following form 
		\begin{align}
		E''_p(t) \geq 2\alpha  \mathcal{F}(0) + 2\alpha  \int_{0}^t \int_{D}  [u_{\tau}^2  + |\nabla_H u|^{p-2}|\nabla_H u_{\tau}|^2]dx d\tau.
		\end{align}
		Also we have for arbitrary $\delta>0$, in view of \eqref{eq-E}, 
		\begin{align*}
		[E'_p(t)]^2&\leq  (1+\delta)\left( \int_{0}^t \frac{d}{d\tau}\int_{D}[u^2 + \frac{2}{p}|\nabla_H u|^p] dx d\tau \right)^2\\
		& + \left( 1+ \frac{1}{\delta}\right)\left( \int_{D} [u^{2}_0 + \frac{2}{p}|\nabla_H u_0|^p] dx \right)^2.
		\end{align*}
		Then by taking $\sigma=\delta=\sqrt{\frac{\alpha}{2}}-1>0$, we arrive at
		\begin{align*}
		&E''_p(t) E_p(t) - (1+\sigma) [E'_p(t)]^2\\
		&\geq  2\alpha M\mathcal{F}(0)+ 2\alpha \left(\int_{0}^t \int_{D}  [u_{\tau}^2  + |\nabla_H u|^{p-2}|\nabla_H u_{\tau}|^2]dx d\tau \right) \left(\int_{0}^t \int_{D} [u^{2} + \frac{2}{p}|\nabla_H u|^p dx] d\tau \right)\\
		&- (1+\sigma) (1+\delta)\left( \int_{0}^t \frac{d}{d\tau}\int_{D}[u^2 + \frac{2}{p}|\nabla_H u|^p] dx d\tau \right)^2 - (1+\sigma) \left( 1+ \frac{1}{\delta}\right)\left( \int_{D} [u^{2}_0 + \frac{2}{p}|\nabla_H u_0|^p] dx \right)^2 \\
		&= 2\alpha M\mathcal{F}(0) - (1+\sigma)\left( 1+ \frac{1}{\delta}\right)\left( \int_{D} [u^{2}_0 + \frac{2}{p}|\nabla_H u_0|^p] dx \right)^2\\
		& + 2\alpha \left[ \left(\int_{0}^t \int_{D}  [u_{\tau}^2  + |\nabla_H u|^{p-2}|\nabla_H u_{\tau}|^2]dx d\tau \right) \left(\int_{0}^t \int_{D} [u^{2} + \frac{2}{p}|\nabla_H u|^p dx] d\tau \right) \right. \\
		&- \left.\left(\int_{0}^t \int_{D} [uu_{\tau} + |\nabla_H u|^{p-2}\nabla_H u \cdot \nabla_H u_{\tau} ]dx d\tau \right)^2 \right]\\
		&\geq 2 \alpha M \mathcal{F}(0) - (1+\sigma)\left( 1+ \frac{1}{\delta}\right)\left( \int_{D} u^{2}_0 + \frac{2}{p}|\nabla_H u_0|^p dx \right)^2.
		\end{align*}
		Note that in the last line we have used the following inequality
		\begin{align*}
		&\left( \int_{0}^t \int_{D}  [u^2  + |\nabla_H u|^p]dx d\tau \right)\left(\int_{0}^t \int_{D}  [u_\tau^2  +|\nabla_H u|^{p-2} |\nabla_H u_\tau|^2]dx d\tau \right) \\
		&-  \left( \int_{0}^t \int_{D}  [u u_{\tau} + |\nabla_{H} u|^{p-2}\nabla_H u \cdot \nabla_H u_{\tau}]dx d\tau  \right)^2 \\
		& \geq \left[ \left( \int_{D} \int_{0}^t u^2 d\tau dx\right)^{\frac{1}{2}}\left( \int_{D} \int_{0}^t |\nabla_H u|^{p-2}|\nabla_H u_{\tau}|^2 d\tau dx\right)^{\frac{1}{2}} \right. \\
		& \left. - \left( \int_{D} \int_{0}^t |\nabla_H u|^p d\tau dx\right)^{\frac{1}{2}}\left( \int_{D} \int_{0}^t  u_{\tau}^2 d\tau dx\right)^{\frac{1}{2}} \right]^2 \geq 0,
		\end{align*}
		where making use of the H\"older inequality and Cauchy-Schawrz inequality we have
		\begin{align*}\small
		&\left( \int_{0}^t\int_{D}  [u u_{\tau} + |\nabla_H u|^{p-2} \nabla_H u\cdot \nabla_H u_{\tau}]dx d\tau  \right)^2 \\
		\leq &\left( \int_{D} \left(\int_{0}^t u^2 d\tau\right)^{\frac{1}{2}} \left(\int_{0}^t u_{\tau}^2 d\tau\right)^{\frac{1}{2}}  dx   +\int_{D} \left(\int_{0}^t |\nabla_H u|^p d\tau\right)^{\frac{1}{2}}\left(\int_{0}^t |\nabla_{H}u|^{p-2}|\nabla_H u_{\tau}|^2 d\tau\right)^{\frac{1}{2}} dx  \right)^2\\
		=& \left( \int_{D} \left(\int_{0}^t u^2 d\tau\right)^{\frac{1}{2}} \left(\int_{0}^t u_{\tau}^2 d\tau\right)^{\frac{1}{2}}  dx\right)^2+ \left(\int_{D} \left(\int_{0}^t |\nabla_H u|^p d\tau\right)^{\frac{1}{2}}\left(\int_{0}^t |\nabla_{H}u|^{p-2}|\nabla_H u_{\tau}|^2 d\tau\right)^{\frac{1}{2}} dx  \right)^2 \\
		+ &2\left( \int_{D} \left(\int_{0}^t u^2 d\tau\right)^{\frac{1}{2}} \left(\int_{0}^t u_{\tau}^2 d\tau\right)^{\frac{1}{2}}  dx\right)\left(\int_{D} \left(\int_{0}^t |\nabla_H u|^p d\tau\right)^{\frac{1}{2}}\left(\int_{0}^t |\nabla_{H}u|^{p-2}|\nabla_H u_{\tau}|^2 d\tau\right)^{\frac{1}{2}} dx  \right)\\
		\leq& \left( \int_{D} \int_{0}^t u^2 d\tau dx\right)\left( \int_{D} \int_{0}^t u^2_{\tau} d\tau dx\right)+ \left( \int_{D} \int_{0}^t |\nabla_H u|^p d\tau dx\right)\left( \int_{D} \int_{0}^t |\nabla_{H}u|^{p-2}|\nabla_H u_{\tau}|^2 d\tau dx\right)\\
		+ &2\left[ \left( \int_{D} \int_{0}^t u^2 d\tau dx\right)\left( \int_{D} \int_{0}^t u^2_{\tau} d\tau dx\right) \left( \int_{D} \int_{0}^t |\nabla_H u|^p d\tau dx\right)\left( \int_{D} \int_{0}^t |\nabla_{H}u|^{p-2}|\nabla_H u_{\tau}|^2 d\tau dx\right)\right]^{\frac{1}{2}}.
		\end{align*}	
		By assumption $\mathcal{F}(0)>0$, thus we can select 
		\begin{equation*}
		M = \frac{(1+\sigma)\left( 1+ \frac{1}{\delta}\right)\left( \int_{D} u^{2}_0 + \frac{2}{p}|\nabla_H u_0|^p dx \right)^2}{2\alpha \mathcal{F}(0)},
		\end{equation*}
		that gives
		\begin{equation}
		E''_p(t) E_p(t) - (1+\sigma) [E'_p(t)]^2 \geq 0.
		\end{equation}
		We can see that the above expression for $t\geq 0$ implies 
		\begin{equation*}
		\frac{d}{dt} \left[ \frac{E'_p(t)}{E^{\sigma+1}_p(t)} \right] \geq 0  \Rightarrow 	\begin{cases}
		E'_p(t) \geq \left[ \frac{E'_p(0)}{E^{\sigma+1}_p(0)} \right] E^{1+\sigma}_p(t),\\
		E_p(0)=M.
		\end{cases}
		\end{equation*}
		Then for $\sigma =\sqrt{\frac{\alpha}{2}}-1>0$, we arrive at 
		\begin{equation*}
		E_p(t) \geq \left( \frac{1}{M^{\sigma}}-\frac{ \sigma \int_{D} [u_0^2 + \frac{2}{p} |\nabla_H u_0|^p ]dx }{M^{\sigma+1}} t\right)^{-\frac{1}{\sigma}}.
		\end{equation*}
		Then the blow-up time $T^*$ satisfies 
		\begin{equation*}
		0<T^*\leq \frac{M}{\sigma \int_{D} [u_0^2 + \frac{2}{p} |\nabla_H u_0|^p] dx}.
		\end{equation*}
		This completes the proof.
	\end{proof}
	
	\subsection{Global solution for the pseudo-parabolic equation} We now show that positive solutions, when they exist for some nonlinearities, can be controlled.   
	\begin{thm}\label{thm_GEp1}
		Let $\mathbb{G}$ be a stratified group with $N_1$ being the dimension of the first stratum. Let $D \subset \mathbb{G}$ be an admissible domain. Let $2 \leq p<\infty$ with $p\neq N_1$.
		
		Assume that function $f$ satisfies
		\begin{equation}\label{G1}
		\alpha F(u) \geq u f(u) + \beta u^{p} +\alpha\gamma, \,\,\, u>0,
		\end{equation}
		where 
		\begin{equation*}
		F(u) = \int_{0}^u f(s)ds,
		\end{equation*}
		for some
		\begin{equation}
		\beta \geq\frac{( p-\alpha)}{2} \,\,\, \text{ and } \,\,\,\alpha \leq 0, \,\, \gamma \geq 0.
		\end{equation} 
		Let $u_0 \in L^{\infty}(D)\cap \mathring{S}^{1,p}(D)$ satisfy 
		\begin{equation}
		\mathcal{F}_0:= -\frac{1}{p} \int_{D} |\nabla_H u_0(x)|^p dx +\int_{D} (F(u_0(x))-\gamma) dx >0. 
		\end{equation}
		If $u$ is a positive local solution of problem \eqref{main_eqn_p}, then it is global and satisfies the following estimate:
		\begin{equation*}
		\int_{D}[ u^{2} + \frac{2}{p}|\nabla_H u|^p] dx \leq \exp({-(p-\alpha)t})\int_{D}[ u^{2}_0 + \frac{2}{p}|\nabla_H u_0|^p] dx.
		\end{equation*}
	\end{thm}
	\begin{proof}[Proof of Theorem \ref{thm_GEp1}]
		Define
		\begin{align*}
		\mathcal{E}(t) :=  \int_{D}[ u^{2} + \frac{2}{p}|\nabla_H u|^p] dx.
		\end{align*}
		Now we estimate $\mathcal{E}'(t) $ by using assumption \eqref{G1}, that gives
		\begin{align*}
		\mathcal{E}'(t)  &= 2\int_{D} u u_t dx +\frac{2}{p} \int_{D} (|\nabla_H u|^p)_t dx\\
		& =2 \int_{D}  [u\mathcal{L}_p u + u\nabla_H \cdot (|\nabla_H u|^{p-2}\nabla_H u_t) + uf(u)]dx + \frac{2}{p} \int_{D} (|\nabla_H u|^p)_t dx\\
		& = -2 \int_{D} [|\nabla_H u|^p + |\nabla_H u|^{p-2}\nabla_H u \cdot \nabla_H u_t ] dx + 2\int_{D} uf(u)dx + \frac{2}{p}\int_{D} (|\nabla_H u|^p)_t dx
		\\
		& \leq 2\alpha \left[ -\frac{1}{p} \int_{D} |\nabla_H u|^p dx + \int_{D} (F(u)-\gamma) dx \right]  - \frac{2(p-\alpha)}{p}  \int_{D} |\nabla_H u|^p dx -2\beta  \int_{D} u^{p} dx \\
		& \leq  2\alpha \left[ -\frac{1}{p} \int_{D} |\nabla_H u|^p dx + \int_{D} (F(u)-\gamma) dx \right] \\
		&- (p-\alpha)[E_p(t) - \int_D u^2 dx]  dx -2\beta  \int_{D} u^{2} dx,\\
		& =2\alpha \mathcal{F}(t) - (p-\alpha)\mathcal{E}(t)  + [p-\alpha -2\beta]\int_{D}u^2 dx,
		\end{align*}
		with
		\begin{align*}
		\mathcal{F}(t)&:= -\frac{1}{p} \int_{D} |\nabla_H u(x,t)|^p dx + \int_{D} (F(u(x,t))-\gamma) dx\\
		&= \mathcal{F}_0 + \int_{0}^t \int_{D}  u_{\tau}^2  + |\nabla_H u|^{p-2}|\nabla_H u_{\tau}|^2dx d\tau.
		\end{align*}
		Since $\beta \geq \frac{p-\alpha}{2}$ we arrive at
		\begin{align*}
		\mathcal{E}'(t)  + (p-\alpha)\mathcal{E}(t)  \leq 2 \alpha \left[ \mathcal{F}_0 +   \int_{0}^t \int_{D}  u_{\tau}^2  + |\nabla_H u|^{p-2}|\nabla_H u_{\tau}|^2dx d\tau \right]\leq 0.
		\end{align*}
		This implies, 
		\begin{equation*}
		\mathcal{E}(t)  \leq \exp({-(p-\alpha)t})\mathcal{E}(0),
		\end{equation*}
		finishing the proof.
	\end{proof}

\end{document}